\documentclass[12pt,a4paper,twoside]{amsart}

\usepackage{amsmath}
\usepackage{amssymb}
\usepackage{amscd,amsthm}
\usepackage{amsmath}
\usepackage{xcolor}
\usepackage{amsthm}
\usepackage{amssymb}
\usepackage{amsfonts}
\usepackage{dsfont}
\usepackage[arrow, matrix, curve]{xy}
\usepackage{mathtools}
\usepackage{longtable}
\definecolor{c1}{RGB}{10,100,155}
\definecolor{c2}{RGB}{50,135,90}
\usepackage[pagebackref=true,colorlinks=true,linktocpage=true,linkbordercolor= c2,citecolor=c1, linkcolor=c2,urlcolor=gray]{hyperref}
\theoremstyle{plain}
\newtheorem{prop}{Proposition}[section]
\newtheorem{theo}[prop]{Theorem}

\newtheorem{constr}[prop]{Mirror Construction}

\theoremstyle{definition}
\newtheorem{dfn}[prop]{Definition}
\newtheorem{rem}[prop]{Remark}
\newtheorem{ex}[prop]{Example}
\newtheoremstyle{noparens}
{}{}
{\itshape}{}
{\bfseries}{{\bf .}}
{ }
{\thmname{#1}\thmnumber{ #2}\mdseries\thmnote{ #3}}
\theoremstyle{noparens}
\newtheorem*{introprop}{Proposition}

\newtheorem{theopar}[prop]{Theorem}

\newtheorem{expar}[prop]{Example}
\usepackage{subcaption}
\usepackage{cleveref}
\DeclareCaptionSubType*[arabic]{figure}
\usepackage{chngcntr}
\counterwithout{subfigure}{section}
\usepackage{threeparttable}

 \usepackage{alphalph}

\usepackage{enumitem}
\DeclareMathOperator{\tope}{\Delta}
\DeclareMathOperator{\vol}{\mathrm{vol}}
\DeclareMathOperator{\Vol}{\mathrm{Vol}}
\DeclareMathOperator{\conv}{\mathrm{conv}}
\DeclareMathOperator{\Hom}{{\mathrm Hom}}

\newcommand*{\defeq}{\coloneqq}
\def\C{{\mathbb C}}
\def\P{{\mathbb P}}
\def\Z{{\mathbb Z}}
\def\N{{\mathbb N}}
\def\R{{\mathbb R}}
\def\Q{{\mathbb Q}}
\def\T{{\mathbb T}}

\def\vol{{\rm vol}}
\def\ggT{{\rm gcd}}

\def\str{{\rm str}}

\def\w{\overline{w}}
\newcommand{\uproman}[1]{\uppercase\expandafter{\romannumeral#1}}

\begin{document}

\title[Calabi-Yau hypersurfaces in weighted projective spaces]{Mirror symmetry 
for quasi-smooth Calabi-Yau hypersurfaces in weighted projective spaces}

\author[Victor Batyrev]{Victor Batyrev}
\address{Mathematisches Institut, Universit\"at T\"ubingen,
Auf der Morgenstelle 10, 72076 T\"ubingen, Germany}
\email{victor.batyrev@uni-tuebingen.de}

\author[Karin Schaller]{Karin Schaller}
\address{Mathematisches Institut, Freie Universit\"at Berlin,
Arnimallee 3, 14195 Berlin, Germany}
\email{karin.schaller@fu-berlin.de}

\begin{abstract}
We consider a $d$-dimensional  well-formed weighted projective 
space $\P(\w)$ as a toric variety associated with a fan $\Sigma(\w)$ in 
$N_{\w} \otimes \R$ whose $1$-dimensional cones are spanned  by primitive 
vectors $v_0, v_1, \ldots, v_d \in N_{\w}$ generating a lattice 
$N_{\w}$ and satisfying the linear relation $\sum_i w_i v_i =0$. 
For any fixed dimension $d$, there exist only finitely 
many weight vectors $\w = (w_0, \ldots, w_d)$ such that $ \P(\w)$ 
contains a quasi-smooth Calabi-Yau hypersurface $X_w$ 
defined by a transverse weighted homogeneous polynomial
$W$ of degree $w = \sum_{i=0}^d w_i$. 
Using a formula of Vafa for the orbifold Euler number  
$\chi_{\rm orb}(X_w)$, we show 
that for any quasi-smooth Calabi-Yau hypersurface $X_w$ 
the number $(-1)^{d-1}\chi_{\rm orb}(X_w)$ equals 
the stringy Euler number  $\chi_{\rm str}(X_{\w}^*)$ of   
Calabi-Yau compactifications $X_{\w}^*$ of
affine toric hypersurfaces $Z_{\w}$ defined by 
non-degenerate Laurent polynomials $f_{\w} \in \C[N_{\w}]$ 
with Newton polytope $\conv(\{v_0, \ldots, v_d\})$. 
In the moduli space of Laurent polynomials $f_{\w}$
there always exists a special point $f_{\w}^0$ defining  a mirror $X_{\w}^*$ 
with a $\Z/w\Z$-symmetry group such that $X_{\w}^*$ is birational to a 
quotient of a Fermat hypersurface via a Shioda map.
\end{abstract}

\maketitle
\thispagestyle{empty}

%%%%%%%%%%%%%%%%%%%%%%%%%%%%%%%%%%%%%%%%%%%%%%%%%%%%%
\section{Introduction}

Many topologically different
examples of smooth Calabi-Yau threefolds and many evidences
for mirror symmetry were obtained from quasi-smooth
Calabi-Yau hypersurfaces $X_w$  in $4$-dimensional
weighted projective spaces
$\P(w_0,w_1,w_2,w_3,w_4)$ defined by weighted homogeneous polynomials 
$W \in \C[z_0, \ldots, z_4]$ of degree $w = \sum_{i=0}^4 w_i$ 
such that the differential 
$dW$ vanishes exactly in the origin $0 \in \C^5$ 
\cite{CLS90,GP90,CdlOGP91,CdlOK95}. 
The Hodge numbers of two $d$-dimensional smooth Calabi-Yau varieties $V$ and $V^*$ that 
are mirror symmetric to each other must satisfy the equalities  
\[ h^{p,q}(V) = h^{d-p,q}(V^*)  \]
for all $0 \leq p,q \leq d$ \cite{Wit92}. In particular, the Euler number 
$\chi = \sum_{p,q} (-1)^{p+q} h^{p,q}$ must satisfy the equality
\[ \chi(V) =  (-1)^d  \chi(V^*).\]
Unfortunately, quasi-smooth Calabi-Yau hypersurfaces 
 $X_w \subset \P(w_0,w_1,w_2,w_3,w_4)$ defined
by weighted homo\-ge\-neous polynomials $W$ are usually singular. 
The first ma\-the\-ma\-tical verifications of the above mentioned equalities for
Hodge and Euler numbers were based on the existence of 
crepant desingularizations $\rho: Y \to X_w$
that allow to replace $X_w$ by  
smooth Calabi-Yau threefolds  $Y$ \cite{Roa90}.
Note that crepant desingularizations  $\rho: Y \to X_w$ 
of quasi-smooth Calabi-Yau hypersurfaces $X_w \subset 
\P(w_0, w_1, \ldots, w_d)$ do not exist in general if  $d \geq 5$. 

\smallskip
Let us recall some definitions and facts on
$d$-dimensional weighted projective spaces $\P(\w)$ 
and quasi-smooth Calabi-Yau hypersurfaces in 
$\P(\w)$, where $\w \defeq (w_0, w_1, \ldots, w_d)$.

\begin{dfn} 
A weighted projective space 
$\P(\w) $ is a quotient of   
$\C^{d+1} \setminus \{0 \}$ by the $\C^*$-action  $\lambda \cdot(a_0, a_1, \ldots, a_d) 
=(\lambda^{w_0} a_0, \lambda^{w_1} a_1, \ldots, \lambda^{w_d} a_d)$ 
defined by the {\em weight vector} $\w=(w_0, w_1, \ldots, w_d) \in \Z^{d+1}_{>0}$, 
whose coordinates $w_i$ are called {\em weights}.  
A weighted projective space $\P(\w)$ is called {\em well-formed} if 
\[\ggT(w_0, \ldots, w_{i-1}, w_{i+1}, \ldots, w_d) =1 \;\; \forall
 i \in I \defeq \{0,1, \ldots, d\}. \] 
In this paper, we consider only well-formed weighted projective spaces $\P(\w)$.  
A weighted homogeneous polynomial $W \in \C[z_0, z_1, \ldots, 
z_d]$ of degree $k$ is characterized by the condition 
\[ W(\lambda^{w_0}z_0, \lambda^{w_1}z_1, \ldots, \lambda^{w_d}z_d) 
= \lambda^k \cdot W(z_0, z_1, \ldots, z_d). \]
Moreover, a weighted homogeneous polynomial $W$ of degree $k$ and the hypersurface 
$X_k \defeq \{ W({\bf z}) = 0 \}  \subset \P(\w)$ of degree $k$
are called {\em transverse} if the common zeros of all partial derivations
\[ \partial W/\partial z_i = 0 \;\;  ( 0 \leq i \leq d)\] is the point $0 \in \C^{d+1}$. 
A weight vector $\w= (w_0, w_1 , \ldots, w_d) \in \Z_{>0}^{d+1}$ is called {\em transverse}  
if there exists at least one transverse Calabi-Yau 
hypersurface $X_w \subset \P(\w)$ of degree 
$w = \sum_i w_i$. Another name, more often used by mathematicians, 
for transverse Calabi-Yau hypersurfaces is {\em quasi-smooth} Calabi-Yau hypersurfaces.
The quasi-smoothness (or transversality) condition 
ensures that the hypersurface $X_w$ has no singularities in addition to those coming
from the singularities of the ambient space \cite{KS98a}, where
the only singularities of weighted projective spaces are
cyclic quotient singularities \cite[Definition 11.4.5]{CLS11}. 
\end{dfn}

\begin{dfn} A weight vector $\w \in \Z^{d+1}_{>0}$ 
has the {\em $\text{IP}$-property} if 
\[ \Delta(W) \defeq \conv(  \{ (u_0, u_1, \ldots, u_d) \in \Z^{d+1}_{\geq 0} 
\, \vert \, \sum_{i=0}^d w_i u_i = w \} ) \subseteq \R^{d+1} \]
is a $d$-dimensional lattice polytope containing the lattice point 
$(1,1,\ldots, 1) \in \Z^{d+1}$ in its interior. 
\end{dfn}

By \cite[Lemma 2]{Ska96}, any transverse weight vector 
$\w \in \Z^{d+1}_{>0}$ has the $\text{IP}$-property. 
For any fixed dimension $d$, there exist only finitely many 
transverse weight vectors and only finitely many 
weight vectors $\w \in \Z_{> 0}^{d+1}$ with $\text{IP}$-property. 
The numbers $T(d)$ and $\text{IP}(d)$ are known for $d \leq 5$: 

\begin{itemize}
\item[-] $T(2) = \text{IP}(2) =3$:  $\{ \w \} = \{(1,1,1), (1,1,2), (1,2,3) \}$;
\item[-]  $T(3) = \text{IP}(3) = 95$   \cite{Rei80};
\item[-]  $T(4) = 7,\!555$ \cite{KS92,KS94}, 
$\text{IP}(4) =184,\!026$ \cite{Ska96};
\item[-]  $T(5) = 1,\!100,\!055$ 
\cite{LSW99,BK16}, 
$\text{IP}(5) =322,\! 383,\! 760,\! 930 $ \cite{SchS19}. 
\end{itemize}

The following statement is a general result of 
Artebani, Comparin, and Guilbot \cite[Theorem 1]{ACG16}
in the particular case of weighted projective spaces:

\begin{prop}
Let $\w = (w_0, w_1, \ldots, w_d)$ be a weight vector 
with $\text{IP}$-property. Then a general hypersurface 
$X_w \subset \P(\w)$ of degree $w = \sum_i w_i$ is 
a Calabi-Yau variety. 
\end{prop}

\smallskip
The present paper is inspired  
by mirror symmetry and the following formula of Vafa for  
the orbifold Euler number  $\chi_{\rm orb}(X_w)$ of 
quasi-smooth Calabi-Yau hypersurfaces 
$X_w  \subset P(\w)=\P(w_0,w_1, \ldots ,w_d)$
of degree $w= \sum_{i=0}^d w_i$ for arbitrary $d\geq 2$:
\begin{align} \label{vafa0}
\chi_{\rm orb}(X_w) = \frac{1}{w} 
\sum_{l,r=0}^{w-1} \, \, \prod_{0 \leq i \leq d \atop lq_i,rq_i \in \Z}
\left(1 - \frac{1}{q_i} \right),
\end{align}
where $q_i \defeq  \frac{w_i}{w}$ for all $i \in I = \{0,1,\ldots,d\}$ and one assumes  
$\prod_{0 \leq i \leq d \atop lq_i,rq_i \in \Z} \left( 1 - \frac{1}{q_i} \right) =1$
if $lq_i,rq_i \not\in \Z$  for all $i \in I$. 
It is remarkable that Formula \eqref{vafa0} appeared without using 
algebraic geometry as the {\em Witten index} ${\rm Tr}(-1)^F$ of the $N=2$ 
superconformal Landau-Ginzburg field theory defined by a weighted homogeneous superpotential $W$ 
of degree $w$ \cite[Formula (28)]{Vaf89}. 

\smallskip
A first general mathematical ($K$-theoretic) interpretation  
of Vafa's Formula \eqref{vafa0} was obtained by Ono and Roan:

\begin{theopar}[{\cite[Theorem 1.1]{OR93}}]  \label{K-th} 
Let $S^{2d+1} \subseteq \C^{d+1} \setminus \{0 \}$ be 
the unit sphere. Consider the compact smooth $(2d-1)$-dimensional real manifold 
 $S_w \defeq S^{2d+1} \cap \{ W = 0\}$  together with the 
 $S^1$-fibration $S_w \to X_w$ which is the restriction of the Seifert $S^1$-fibration 
$S^{2d+1} \to \P(w_0,w_1, \ldots ,w_d)$ to a quasi-smooth Calabi-Yau hypersurface 
$X_w$ of degree $w = \sum_{i=0}^d w_i$. 
Then the $S^1$-equivariant $K$-groups $K^i_{S^1}(S_w)$  $(i=0,1)$ 
have finite rank and 
\[ {\rm rank}\,  K^0_{S^1}(S_w) -  {\rm rank}\,  K^1_{S^1}(S_w) = 
\frac{1}{w} \sum_{l,r=0}^{w-1} \, \, \prod_{0 \leq i \leq d \atop lq_i,rq_i \in \Z}
\left( 1 - \frac{1}{q_i} \right). \]
In particular, 
the right side of this equation is an integer.
\end{theopar}

Quasi-smooth Calabi-Yau hypersurfaces 
$X_w \subset \P(\w)=\P(w_0,w_1, \ldots ,w_d)$ can be considered 
as orbifolds, \emph{i.e.}, as geometric objects that are locally quotients
of smooth manifolds by finite group actions. The  
orbifold cohomology theory developed by Chen and Ruan \cite{CR04,ALR07} 
allows to define for any projective orbifold $V$ certain vector spaces 
$H^{p,q}_{CR}(V)$, whose dimensions $h^{p,q}_{\rm orb}(V)$
are called {\em orbifold Hodge numbers}. In particular, one obtains 
\[ \chi_{\rm orb}(V) \defeq \sum_{p,q} (-1)^{p+q} h^{p,q}_{\rm orb}(V). \]
Another mathematical approach to orbifold Hodge and orbifold Euler 
numbers was suggested  in \cite{BD96,Bat98} via so-called 
{\em stringy Hodge numbers} 
$h^{p,q}_{\rm str}(X)$ that were defined for arbitrary 
projective algebraic varieties $X$ with at worst 
Gorenstein canonical singularities. 
We note that such singularities do not necessarily 
admit a local orbifold structure.  
The mathematical definition of stringy Hodge numbers $h^{p,q}_{\rm str}(X)$ 
uses an arbitrary desingularization $\rho : Y \to X$. Their definition  
immediately implies the equalities
$h^{p,q}_{\rm str}(X) =h^{p,q}(Y)$ 
if the desingularization  $\rho$ is crepant \cite{Bat98}.  

\smallskip
Our consideration of the orbifold Euler number   
is motivated by the combinatorial mirror duality 
for Calabi-Yau hypersurfaces $X$ in $d$-dimensional 
Gorenstein toric Fano varieties  \cite{Bat94}.  
This mirror duality is based on the  
polar duality $\Delta \leftrightarrow \Delta^*$ between $d$-dimensional 
reflexive lattice polytopes $\Delta$ and 
$\Delta^*$. Recall that a $d$-dimensional lattice polytope 
$\Delta \subseteq M_\R \cong \R^d$  
is called {\em reflexive} if it contains the lattice point 
$0 \in M$ in its interior and the {\em dual polytope} 
\[ \Delta^* \defeq \{ y \in N_\R \, \vert\, 
\langle x, y \rangle \geq -1 \;\; \forall x \in \Delta \} \subseteq N_\R \]
is also a lattice polytope. If $\Delta $ is reflexive, then 
$\Delta^*$ is also reflexive and $(\Delta^*)^* = \Delta$. 
The polar duality induces a natural
one-to-one correspondence $\theta \leftrightarrow \theta^*$ 
between $k$-dimensional faces $\theta \preceq\ \Delta$ of $\Delta$ and 
$(d-k-1)$-dimensional faces $\theta^* \preceq \Delta^*$ of $\Delta^*$. 
If $X$ is a general Calabi-Yau hypersurface in the Gorenstein toric Fano variety 
$\P_\Delta$ associated with the normal fan $\Sigma_{\tope}$ of $\Delta$, 
then the {\em stringy Euler number}
of $X$ can be computed by the following combinatorial 
formula \cite[Cor. 7.10]{BD96}: 
\begin{align}  \label{comb-Euler}
\chi_{\rm str}(X)= \sum_{k=1}^{d-2} (-1)^{k-1} \sum_{ \theta \preceq \Delta \atop \dim(\theta) = k} 
\Vol_k (\theta) \cdot \Vol_{d-k-1} (\theta^*),  
\end{align} 
where $\Vol_{\dim(\cdot)}(\cdot) \in \N$ denotes the normalized volume (cf. Definition \ref{normvol}). 
Using Formula \eqref{comb-Euler} and the bijection $\theta \leftrightarrow \theta^*$, 
 one immediately obtains the equality 
\begin{align*} 
\chi_{\rm str}(X) = (-1)^{d-1} \chi_{\rm str}(X^*), 
\end{align*}
where $X^*$ denotes the mirror of $X$ obtained as a 
Calabi-Yau hypersurface in the Gorenstein toric Fano 
variety $\P_{\Delta^*}$. Moreover, by combinatorial 
methods one can show the following equalities of stringy Hodge numbers 
\[h^{p,q}_{\rm str}(X) = h^{d-1-p,q}_{\rm str}(X^*) \;\; (0 \leq p,q \leq d-1) \]
predicted by mirror duality \cite{BB96}. 

\begin{rem}
The following combinatorial property of the polar duality 
for reflexive polytopes is crucial for the combinatorial  Mirror Construction \ref{conj-mirr}: 

\begin{itemize} 
\item[] The set of  generators of  $1$-dimensional cones $\sigma$ in the normal fan 
$\Sigma_\Delta \subseteq N_\R$ of a reflexive polytope $\Delta \subseteq M_\R$ 
is the set of vertices of the dual reflexive polytope 
$\Delta^* \subseteq N_\R$.  
\end{itemize}
\end{rem}

We consider a well-formed weighted projective space  $\P(\w)$ as 
a projective toric variety corresponding to a $d$-dimensional 
simplicial fan $\Sigma(\w)$ in $N_{\w} \otimes \R$ whose $1$-dimensional
cones are spanned by primitive lattice vectors 
$v_0, v_1, \ldots, v_d \in N_{\w}$ generating the lattice $N_{\w}$ and satisfying the linear 
relation $\sum_{i=0}^d w_i v_i =0$. In this way, the lattice 
$N_{\w}$ can be identified with the quotient lattice $\Z^{d+1}/\Z\w$. 
 
 \smallskip
 If $\w \in \Z_{>0}^{d+1}$ is a weight vector having the $\text{IP}$-property, then 
 the $d$-dimensional lattice polytope
 $\Delta'(W) \defeq  \Delta(W) - (1, \ldots, 1) \subseteq M_{\w} \otimes \R$ 
 contains the origin $0 \in M_{\w}$ in its interior. 
We define the {\em spanning fan}  $\Sigma^{\vee}(\w)$ of $\Delta'(W)$ as 
 the fan  in $M_{\w} \otimes \R$  consisting  of all cones $\R_{\geq 0}\theta'$ 
 over faces $\theta' \preceq \Delta'(W)$ of  the lattice polytope $\Delta'(W)$. 
 By the theory of toric varieties \cite{CLS11},  the fan  $\Sigma^{\vee}(\w)$ defines 
 a $d$-dimensional $\Q$-Gorenstein toric Fano variety $\P^{\vee}(\w)$, which is a 
compactification of the algebraic torus $\T_{\w}$ with the lattice of 
characters $N_{\w}$, \emph{i.e.},
 \[\T_{\w} \defeq \{(x_0,x_1,\ldots,x_d) \in  (\C^*)^{d+1} \, \vert \,  \prod_{i=0}^d x_i^{w_i} =1 \}
\subset  (\C^*)^{d+1}. \]
 
Using several results from \cite{Bat17}, 
we prove the following main statement of our paper:
 
\begin{theo} \label{vafemirrortest}
Let  $\w = (w_0, w_1, \ldots, w_d)$ be a weight vector with $\text{IP}$-property
and $Z_{\w} \subset \T_{\w}$ a non-degenerate affine
hypersurface defined by a Laurent 
polynomial $f_{\w}$ with Newton polytope
 $\Delta_{\w}^*= \conv(\{v_0, \ldots, v_d\}) \subseteq N_{\w} \otimes \R$. 
Then the Zariski closure of $Z_{\w}$ 
in the $\Q$-Gorenstein toric Fano variety $\P^{\vee}({\w})$
is a Calabi-Yau hypersurface $X_{\w}^*$ and  
\[  \chi_{\rm str}(X^*_{\w})  =  (-1)^{d-1} \frac{1}{w} \sum_{l,r=0}^{w-1} \, \, 
\prod_{0 \leq i \leq d \atop lq_i,rq_i \in \Z} \left( 1 - \frac{1}{q_i} \right), \]
where $q_i = \frac{w_i}{w}$ $($$i \in I$$)$. 
In particular, if $\w$ is a transverse weight vector, \emph{i.e.}, there exists
a quasi-smooth Calabi-Yau hypersurface 
$X_w \subset \P(\w)$ of degree $w = \sum_i w_i$, then 
\[ \chi_{\rm str}(X_{\w}^*) = (-1)^{d-1} \chi_{\rm orb}(X_w). \]
\end{theo}
 
The last statement in Theorem \ref{vafemirrortest} 
supports the following mirror construction:

\begin{constr} \label{conj-mirr}
Let $\w \in \Z_{>0}^{d+1}$ be an arbitrary transverse 
weight vector. Then mirrors of quasi-smooth Calabi-Yau hypersurfaces 
$X_w\subset \P(\w)$ can be obtained as Calabi-Yau compactifications 
$X_{\w}^*$ of non-degenerate affine hypersurfaces $Z_{\w} \subset \T_{\w}$ 
with Newton polytope $\Delta_{\w}^* = \conv(\{v_0, v_1, \ldots, v_d\})$.
\end{constr}

\begin{rem}
If the weighted projective space 
$\P(\w)=\P(w_0, w_1, \ldots, w_d)$ is not Gorenstein, then 
the lattice simplex $\Delta_{\w}^*$ is not reflexive and its 
dual simplex 
\[\Delta_{\w} =  \{ (u_0, u_1, \ldots, u_d) \in \R^{d+1}_{\geq 0} 
\, \vert \, \sum_{i=0}^d w_i u_i = w \}  - (1,\ldots, 1)\]
is a rational simplex containing 
the lattice polytope $\Delta'(W) \subseteq M_{\w} \otimes \R$. It  
follows from \cite[Theorem 2]{ACG16} that if a $\Q$-Gorenstein 
weighted projective space $\P(\w)$ contains
a Calabi-Yau hypersurface $X_w$ defined by an 
{\em invertible} weighted homogeneous polynomial 
$W$, then the lattice simplex $\Delta_{\w}^*$ is the Newton polytope
of the Berglund-H\"ubsch-Krawitz mirror of $X_w \subset \P(\w)$
\cite[Section 5]{CR18}, \cite{Kra10,CR11,Bor13}.
\end{rem}

\begin{rem}
Note that a weighted projective space $\P(\w)= \P(w_0,w_1, \ldots ,w_d)$ 
is a Gorenstein if and only if 
each weight $w_i \; ( 0 \leq i \leq d)$ divides the degree 
$w = \sum_{i=0}^d w_i$. In the latter case, one can 
choose a  quasi-smooth Calabi-Yau hypersurface  $X_w$ 
defined by the following weighted homogeneous polynomial 
of {\em Fermat-type}:
\[ W= z_0^{w/w_0} + z_1^{w/w_1} + \cdots + z_d^{w/w_d}
\in \C[z_0, z_1, \ldots, z_d]. \]
In this case, the two simplices $\Delta_{\w}^*$ and 
\[ \Delta_{\w} = \Delta'(W) =\Delta(W) - (1, \ldots, 1) \]
are reflexive and Mirror Construction \ref{conj-mirr} 
coincides with the orbifoldizing mirror 
construction of Greene-Plesser \cite{GP90}.
\end{rem}

\begin{rem}
The lattice simplex $\Delta_{\w}^* = \conv(\{v_0, v_1, \ldots, v_d\})$ 
is the Newton polytope of the special non-degenerate Laurent polynomial 
$f_{\w}^0({\bf t}) \defeq \sum_{i=0}^d {\bf t}^{v_i}  \in \C[N_w]$
defining the affine hypersurface 
\[Z_{\w}^0 = \T_{\w} \cap 
 \{(x_0,x_1,\ldots,x_d) \in (\C^*)^{d+1}   \,  \vert \, \sum_{i=0}^d x_i =0  \} 
\subset \T_{\w},\] 
\emph{i.e.},  $f^0_{\w}$ is the restriction of the linear function 
$x_0 + x_1 + \ldots + x_d: (\C^*)^{d+1} \to \C $ 
to the $d$-dimensional subtorus $\T_{\w} \subset  (\C^*)^{d+1}$. 
Note that the Laurent polynomial $f_{\w}^0$ appears in the Givental-Hori-Vafa mirror 
construction for the weighted projective space $\P(\w)$ \cite[Section 4.4]{CR18}, \cite{HV00}. 
The restriction  of the projection 
\[ (\C^*)^{d+1} \to (\C^*)^d, \;\; (x_0, x_1, \ldots, x_d) \mapsto
(x_1/x_0, \ldots, x_d/x_0) \]
to the affine hypersurface 
$Z_{\w}^0 \subset \T_{\w} \subset  (\C^*)^{d+1}$
defines an unramified cyclic  Galois covering      
$\gamma_w : Z_{\w}^0 \to Z_0$ of order $w$, where 
\[ Z_0 \defeq \{ (y_1, \ldots, y_d) \in (\C^*)^d \, \vert \,  
\sum_{i=1}^d y_i = -1 \} \subset (\C^*)^d. \] 
Therefore, the affine hypersurface $Z_{\w}^0$ can be obtained 
as a $(\Z/w\Z)^{d-1}$-quotient of the affine
Fermat hypersurface 
\[ F_w \defeq  \{ (y_1, \ldots, y_d) \in (\C^*)^d \, \vert \,  
\sum_{i=1}^d y_i^w = -1 \} \subset (\C^*)^d. \] 
Thus, the Givental-Hori-Vafa polynomial $f^0_{\w}$ defines a special 
"Fermat-type" point in the moduli space 
of mirrors of quasi-smooth 
Calabi-Yau hypersurfaces $X_w \subset \P(\w)$. 
The Givental-Hori-Vafa mirrors  $Z^0_{\w}$ were considered by Kelly \cite{Kel13} via 
so-called {\em Shioda maps} in the Berglund-H\"ubsch-Krawitz mirror constructions 
for Calabi-Yau hypersurfaces $X_w \subset \P(\w)$ defined by invertible 
polynomials $W$. 
\end{rem}
 
\begin{expar}[{\cite[Example 53]{CR18}}] 
Let   $\w =(w_0, w_1, \ldots, w_d)$ be a weight vector 
with $w_0 =1$. Then   $v_0 = - \sum_{i=1}^d w_i$ and 
the lattice vectors $v_1, \ldots, v_d \in \Z^d$ form 
a $\Z$-basis of $N_{\w}$. Thus, the Givental-Hori-Vafa 
polynomial $f^0_{\w} \in \C[N_{\w}]$ 
can be written  in the form
\[ f^0_{\w}({\bf t}) = \sum_{i=0}^d {\bf t}^{v_i} = 
\frac{1}{t_1^{w_1} \cdot \ldots \cdot t_d^{w_d}}+  t_1 + \ldots + t_d. \] 
\end{expar}

We note that we can not expect Mirror Construction \ref{conj-mirr}  
working for Calabi-Yau hypersurfaces $X_w  \subset \P(\w)$ 
that are not quasi-smooth. The following Examples \ref{not-a-mirror} 
and \ref{EulerQ} were proposed by Harald Skarke:

\begin{expar}[{\cite[Example 2]{KS92}}]  \label{not-a-mirror}
Consider the $\text{IP}$-weight vector  $\w =(1,1,6,14,21)$. 
One can show that  the weight vector $\w$
is not transverse. The $4$-dimensional lattice polytope
$\Delta'(W) \subset  \Delta_{\w}$ 
is reflexive \cite[Lemma 1]{Ska96}. 
Therefore, general Calabi-Yau hypersurfaces $X_{43} \subset \P(\w)$ 
are birational to smooth Calabi-Yau
threefolds $Y$ with  Hodge numbers 
$h^{1,1}(Y) = 21$, $h^{2,1}(Y) = 273$,  
and Euler number $\chi(Y) = - 504$   \cite{Bat94}.

\smallskip
On the other hand, by Theorem \ref{vafemirrortest},  
the Givental-Hori-Vafa hypersurface
$Z_{\w}^0 \subset (\C^*)^4$ defined by the equation 
\[ \frac{1}{t_1 t_2^6 t_3^{14} t_4^{21}} + t_1 +t_2 + t_3 + t_4 = 0  \]
is birational to a $3$-dimensional Calabi-Yau variety $X^*_{\w}$ 
with  stringy Euler number 
\[\chi_{\rm str}(X^*_{\w}) = - \frac{1}{43} 
\sum_{l,r=0}^{42} \prod_{0 \leq i \leq 4 \atop lq_i, rq_i \in \Z}
\left( 1 - \frac{1}{q_i} \right) =     506  \neq 504 = 
 - \chi_{\rm str}(X_{43}) = -\chi(Y).\]
Therefore, $X^*_{\w}$ is not mirror of $X_{43}$. 
\end{expar}

\begin{ex}\label{EulerQ} 
The weight vector  $\w =(1,1,2,4,5)$ has {\rm IP}-property.  
Since the  $4$-dimensional lattice polytope
$\Delta'(W) \subset  \Delta_{\w}$ is reflexive, 
a general Calabi-Yau hypersurface $X_{13} \subset \P(\w)$ 
is birational to a smooth Calabi-Yau 
threefold. On the other hand, 
by Theorem \ref{vafemirrortest},  
the Givental-Hori-Vafa hypersurface
$Z_{\w}^0 \subset (\C^*)^4$ defined by the equation 
\[ \frac{1}{t_1 t_2^2 t_3^{4} t_4^{5}} + t_1 +t_2 + t_3 + t_4 = 0  \]
is birational to a $3$-dimensional Calabi-Yau variety $X^*_{\w}$ 
with  stringy Euler number 
\[\chi_{\rm str}(X^*_{w}) = - \frac{1}{13} 
\sum_{l,r=0}^{12} \prod_{0 \leq i \leq 4 \atop lq_i, rq_i \in \Z}
\left( 1 - \frac{1}{q_i} \right) =   \frac{1032}{5} \not\in \Z. \]
Therefore, the Calabi-Yau variety $X^*_{\w}$ cannot have a  Landau-Ginzburg 
description and  $X^*_{\w}$ has no 
mirror at all. 
\end{ex}

%%%%%%%%%%%%%%%%%%%%%%%%%%%%%%%%%%%%%%%%%%%%
\section{Stringy Euler numbers of toric Calabi-Yau hypersurfaces}

Let $M \cong \Z^d$ be a lattice of rank $d$, $N \defeq \Hom(M, \Z)$ 
its dual lattice, and 
\[\langle \cdot, \cdot \rangle : M \times N \to \Z\]
the natural pairing. Consider a $d$-dimensional lattice polytope 
$\tope \subseteq M_\R \defeq M \otimes_{\Z} \R \cong \R^d$ and
denote by $\Sigma_{\tope}$ the normal fan of $\tope$ in the dual real vector
space $N_\R \defeq N \otimes_{\Z} \R$. There exists a natural
one-to-one correspondence $\theta \leftrightarrow \sigma_{\theta}$  between 
$k$-dimensional faces $\theta \preceq \tope$ and $(d-k)$-dimensional
cones $\sigma_{\theta} \in \Sigma_{\tope}$, where 
\[ \sigma_{\theta}= \{ y \in N_\R\, | \, 
\min_{x \in \Delta} \langle x, y \rangle  =   \langle x', y \rangle  \; \; \forall x' \in \theta \}. \]   

\smallskip
We are interested in $d$-dimensional 
Newton polytopes $\Delta \subseteq M_\R$ of Laurent polynomials $f_{\tope} \in \C[M]$ 
defining non-degenerate affine hypersurfaces 
\[Z_{\tope}  \subset \T \defeq \Hom(M,\C^*) \cong (\C^*)^d\] 
that are birational to {\em Calabi-Yau varieties}
$X$, \emph{i.e.}, normal projective algebraic varieties with at worst 
canonical singularities and trivial canonical class. 
By \cite[Theorem on page 41]{Kho78}, such a Newton  
polytope $\Delta$ must be a {\em canonical Fano polytope}, 
\emph{i.e.}, a lattice polytope containing exactly one interior lattice 
point that we can assume to be $0 \in M$. 

\smallskip
We review two results from \cite{Bat17}.
The first is a combinatorial 
characterization of $d$-dimensional canonical  Fano polytopes $\tope$ 
such that $Z_{\tope}$ is birational to a  
Calabi-Yau hypersurface $X$:

\begin{theopar}[{\cite[Theorem 2.23]{Bat17}}]  \label{Bat17223}
Let $\Delta \subseteq M_\R$ be a $d$-dimensional canonical Fano polytope 
and $\tope^* = \{ y \in N_\R\; \vert \; \langle x, y \rangle \geq -1 \;\; \forall 
x \in \tope \}$ its rational dual polytope. We set 
\[[\tope^*] \defeq \conv( \tope^* \cap \; N) \subseteq N_\R \]
to be the convex hull  of all lattice points in $\tope^*$. 
Then a non-degenerate affine hypersurface $Z_{\tope} \subset \T$ 
is birational to a Calabi-Yau variety $X$ if and only if $[\tope^*]$ 
is also a $d$-dimensional canonical Fano polytope.  
\end{theopar}

\begin{rem} \label{ifpart}
The "if"-part in Theorem \ref{Bat17223} was independently proven by 
Artebani, Comparin, and Guilbot \cite[Theorem 1]{ACG16}.  
Moreover, they proved that a canonical Calabi-Yau model $X$ 
of the hypersurface $Z_\Delta \subset \T$ can be explicitly 
constructed, \emph{e.g.},  as the Zariski closure 
of $Z_\Delta$ in the toric $\Q$-Fano variety 
associated with the spanning fan of the canonical Fano polytope 
$[\Delta^*] \subseteq N_\R$.
\end{rem}

\begin{rem}
Canonical Fano polytopes $\tope \subseteq M_\R$ 
satisfying the criterion  in Theorem \ref{Bat17223} 
can be equivalently characterized by the condition 
that the Fine interior of $\tope$ is exactly its interior 
lattice point   
$\{0\}$ \cite[Section 2]{Bat17}, \cite{BKS19}. 
\end{rem}

\begin{rem}
A $d$-dimensional canonical Fano polytope $\Delta$ is called 
{\em pseudoreflexive} if it satisfies the condition 
\[ [[\Delta^*]^*] = \Delta.\]
It is easy to see that if $\Delta$ is pseudoreflexive, then 
$[\Delta^*]$ is also pseudoreflexive. Thus, pseudoreflexive 
polytopes satisfy the combinatorial duality $\Delta \leftrightarrow [\Delta^*]$, 
which was proposed by Mavlytov as a generalization of
the polar duality for reflexive polytopes \cite{Mav11}. We note 
that any pseudoreflexive polytope $\Delta$ of dimension 
$d \leq 4$ is reflexive. A simple example of 
a $5$-dimensional pseudoreflexive polytope $\Delta$ 
that is not reflexive can be obtained 
as the Newton polytope of a general $4$-dimensional Calabi-Yau hypersurface 
of degree $7$ in the weighted projective space $\P(1,1,1,1,1,2)$ \cite{KS98b}.  
It was expected by Mavlyutov that if  $(\Delta, [\Delta^*])$ is 
a pair of pseudoreflexive polytopes, then 
Calabi-Yau hypersurfaces in toric 
$\Q$-Fano varieties corresponding to rational 
polytopes $\Delta^*$ and $[\Delta^*]^*$   
are mirror symmetric to each other.  For Calabi-Yau 
hypersurfaces of degree $7$ in $\P(1,1,1,1,1,2)$ this 
is true, but in general this is false \cite[Section 5]{Bat17}. 
\end{rem}

\begin{rem} \label{almostpseudo}
Let $\Delta$ be an arbitrary canonical Fano polytope satisfying the 
condition in Theorem \ref{Bat17223}, 
\emph{i.e.}, $[\Delta^*]$ is canonical. In general, $\Delta$ 
need not to be pseudoreflexive.  However, 
the canonical polytopes $[\Delta^*]$ 
and $[[\Delta^*]^*]$ are always pseudoreflexive. For this 
reason, we call a canonical Fano polytope $\Delta$  
{\em almost pseudoreflexive} if  
$[\Delta^*]$ is also a canonical Fano polytope. 
The name is motivated by the fact that in this case 
$[[\Delta^*]^*]$ is the smallest pseudoreflexive 
polytope containing $\Delta$. So it is natural to call $[[\Delta^*]^*]$ 
the {\em pseudoreflexive hull} of $\Delta$. Since 
pseudoreflexive polytopes $\Delta$ of dimension $3$ and $4$ 
are reflexive, we will use 
the expressions  {\em almost reflexive} and {\em reflexive hull} 
instead of {\em almost pseudoreflexive} and {\em pseudoreflexive hull}, 
respectively. 
\end{rem} 

\begin{rem}
Recall some classification results concerning  
special classes of canonical Fano polytopes:
\begin{itemize}
\item Any $2$-dimensional 
canonical Fano polytope is reflexive. 
There exist exactly 16 of them. 
\item All $3$-dimensional canonical Fano polytopes 
are classified by Kasprzyk \cite{Kas10}. There exist exactly  $665,\!599$ 
$3$-dimensional almost reflexive polytopes among all $674,\!688$ 
$3$-dimensional canonical Fano polytopes. This list extends the classification of 
all $4,\!319$ reflexive $3$-dimensional polytopes obtained by Kreuzer and 
Skarke \cite{KS98a}. 
\item All $4$-dimensional reflexive polytopes are classified by 
Kreuzer and Skarke \cite{KS00}. This list consists of $473,\!800,\!776$ reflexive 
polytopes, but the list of all $4$-dimensional 
almost reflexive polytopes is still unknown. 
\end{itemize}
\end{rem}

The second result from \cite{Bat17} is a combinatorial formula 
for the stringy Euler number  $\chi_{\rm str} (X)$ of a  Calabi-Yau 
variety $X$ generalizing Formula \eqref{comb-Euler} 
for reflexive polytopes.  
We recall a general definition for the stringy Euler number
of a $d$-dimensional normal projective $\Q$-Gorenstein variety 
$X$ with at worst log-terminal singularities: 

\begin{dfn}
Let $\rho :Y \to X$ be a log-desingularization of $X$ whose 
exceptional locus is the union of smooth irreducible divisors $D_1, 
\ldots, D_s$ with only simple normal crossing and 
\[ K_Y = \rho^* K_X + \sum_{i =1}^s a_i D_i \] for some rational numbers 
$a_i > -1$ $(1 \leq i \leq s)$. We set  $D_\emptyset \defeq Y$ and 
\[D_J \defeq \bigcap_{j \in J} D_j, \; \;\; \emptyset \neq 
J \subseteq I\defeq \{1, \ldots, s\}.\]
Then the {\em stringy Euler number} $\chi_{\rm str}(X)$ 
is defined to be 
\[ \chi_{\rm str}(X) \defeq \sum_{ \emptyset \subseteq J \subseteq I}
\chi(D_J) \prod_{j \in J} \left(  \frac{-a_j}{a_j +1} \right),  \]
where $\chi(D_J)$ denotes the usual Euler number of 
the smooth projective variety $D_J$ \cite{Bat98}. 
\end{dfn}

\begin{rem}
It is important to note that the stringy Euler number is  
independent on the log-desingularization $\rho: Y \to X$
and $\chi_{\rm str}(X)$ equals the usual Euler number of $Y$ 
if $\rho: Y \to X$ is crepant. More generally, if two 
projective algebraic varieties $X$ and $X'$ are $K$-equivalent, 
then $\chi_{\rm str}(X) = \chi_{\rm str}(X')$. In particular, 
the stringy Euler number of a given birational class 
of algebraic varieties of non-negative Kodaira dimension is well-defined. 
\end{rem}

In addition, we will need the notion of the  
normalized volume of a rational polytope:

\begin{dfn} \label{normvol}
Let $\Delta \subseteq N_{\R}$ be a $d$-dimensional {\em rational polytope}, \emph{i.e.},
$\Delta$ has vertices in $N_{\Q} \defeq N \otimes \Q$. 
Then  the positive rational number 
\[\Vol_d(\Delta) \defeq \frac{1}{l^d} \Vol_d(l \Delta)= \frac{1}{l^d} \cdot d!\cdot \vol_d (l \Delta)\]
is called the {\em normalized volume} of $\Delta$,  
where $l$ is a positive integer such that $l\Delta$ is 
a lattice polytope and $\vol_d(\cdot)$ denotes the $d$-dimensional volume of 
$\Delta$ with respect to the lattice $N$. Similarly, one obtains 
the normalized volume $\Vol_k(\theta) \defeq 
\frac{1}{l^k} \Vol_k(l\theta)$ with respect to the sublattice $\text{span}(\theta) \cap N$ 
for a $k$-dimensional rational face
$\theta \preceq \Delta$.
\end{dfn}

\begin{theopar}[{\cite[Theorem 4.11]{Bat17}}] \label{Bat17411}
Assume that a $d$-dimensional canonical Fano polytope $\tope \subseteq M_\R$ 
satisfies the condition in Theorem \ref{Bat17223}, \emph{i.e.}, 
a non-degenerate affine hypersurface $Z_\Delta \subset \T$ 
is birational to a Calabi-Yau variety $X$. 
Then the stringy Euler number of $X$ can be computed 
by the following combinatorial formula
\begin{align*}  
\chi_{\rm str} (X)  = \sum_{k=1}^{d} (-1)^{k-1}  \sum_{ \theta \preceq \Delta \atop \dim (\theta) = k} 
\Vol_k (\theta) \cdot \Vol_{d-k} (\sigma_\theta \cap \Delta^*),  
\end{align*}
where $\sigma_\theta \cap \Delta^*$ is the $(d-k)$-dimensional 
pyramid with vertex $0 \in N$ over the $(d-k-1)$-dimensional dual face $\theta^* \preceq \Delta^*$
of the dual rational polytope $\Delta^*$. 
\end{theopar}

\begin{ex}
Let $\tope \subseteq M_\R$ be a $3$-dimensional 
almost reflexive polytope. Then $X$ 
is a $K3$-surface with  $\chi_{\rm str} (X)=24$. In this case, the 
combinatorial formula in Theorem 
\ref{Bat17411} is equivalent to the identity 
\[ 24 = \Vol_3(\Delta) - \sum_{ \theta \prec \Delta \atop 
\dim (\theta) = 2} \frac{1}{n_\theta} \cdot
\Vol_2(\theta)  + \sum_{ \theta \prec \Delta \atop 
\dim (\theta) = 1} \Vol_1(\theta) \cdot \Vol_1(\theta^*), \]
where $n_\theta$ denotes the integral distance 
between a facet $\theta \preceq \Delta$ and $0 \in M$
\cite[Theorem A]{BS19}.   
\end{ex}

\begin{ex}
Let  $\tope \subseteq M_\R$ be a $d$-dimensional reflexive polytope. Then 
the dual polytope $\tope^*$ is also a reflexive polytope. 
Therefore, we obtain 
\[ \Vol_{d-k}(\sigma_\theta \cap \Delta^*) = 
\Vol_{d-k-1}(\theta^*)\] 
for any $k$-dimensional face $\theta \preceq \tope$ of $\tope$ $(1 \leq k \leq d-2)$. 
Together with the equality 
$ \Vol_d(\Delta) = \sum_{ \theta \preceq \Delta \atop \dim (\theta) = d-1} 
\Vol_{d-1} (\theta)$, Theorem \ref{Bat17411} implies the already discussed 
combinatorial Formula \eqref{comb-Euler}.
\end{ex}

%%%%%%%%%%%%%%%%%%%%%%%%%%%%%%%%%%%%%%%%%%%%%%%%%%%%
\section{Calabi-Yau hypersurfaces in weighted projective spaces}

Let $(w_0, w_1, \ldots, w_d) \in \Z_{>0}^{d+1}$ be a 
well-formed weight vector.  
In this section, we apply Theorems \ref{Bat17223} and \ref{Bat17411}
to the special  $d$-dimensional lattice simplex 
\[\tope \defeq  \tope_{\w}^* = \conv(\{v_0, v_1, \ldots, v_d\}) \subseteq  N_{\w} \otimes \R,\] 
where $v_0, \ldots, v_d$ generate the lattice
$M  \defeq N_{\w} $ and satisfy the linear relation $\sum_{i=0}^d w_i v_i = 0$. 
Consider the lattice $\widetilde{M} \defeq M \oplus \Z$ together with $d+1$ 
linearly independent lattice vectors $\widetilde{v}_i \defeq (v_i,1) \in \widetilde{M}$ 
($0 \leq  i \leq d$). Set $M'\defeq \langle  \widetilde{v}_i \rangle_{0 \leq i \leq d}$ and 
identify $M'$ with $\Z^{d+1}$ via the basis $\widetilde{v}_i$ ($0 \leq  i \leq d$). 
Since $M$ is generated by ${v}_i$ ($0 \leq  i \leq d$), the lattice $\widetilde{M}$ 
is generated by $\widetilde{v}_i $ together with the lattice vector 
$\widetilde{v} \defeq (0, 1) \in M \oplus \Z= \widetilde{M}$ and 
\[ \widetilde{v} = \sum_{i=0}^d \frac{w_i}{w} \cdot \widetilde{v}_i. \]
So the quotient $\widetilde{M} / M'$ is a cyclic
group of order $w$ generated by $\widetilde{v} + M'$ and we can write
\[ \widetilde{M} =\Z \left( \frac{w_0}{w},  \frac{w_1}{w}, 
\ldots,  \frac{w_d}{w}\right) +  \Z^{d+1} = 
\Z( q_0, q_1, \ldots, q_d) + \Z^{d+1}. \]
In certain cases, it will be convenient to use the multiplicative 
description of $\widetilde{M} / M'$ by the cyclic group 
$G \defeq \langle g \rangle  = \{g^s \, |\, s \in \Z/w\Z \} \subseteq  SL(d+1, \C)$ 
generated by the diagonal matrix 
\[ g \defeq {\rm diag}(g_0, g_1, \ldots, g_d) 
\defeq {\rm diag}(e^{2\pi i q_0}, e^{2\pi i q_1}, \ldots, e^{2\pi i q_d}). \] 
In the language of  toric varieties, the $(d+1)$-dimensional 
simplicial cone $\R_{\geq 0}(1, \Delta) \subseteq \widetilde{M}_\R$ describes the 
Gorenstein cyclic quotient singularity $\C^{d+1}/G$ considered
by Corti and Golyshev \cite{CG11}. 

\begin{dfn}
For any non-empty subset $J \subseteq I$, we define two 
$\vert J \vert$-dimensional sublattices 
\[ M_J' \defeq \sum_{j \in J} \Z \widetilde{v_j} \; \; \text{ and } \;\; 
\widetilde{M}_J \defeq \widetilde{M} \cap (M_J' \otimes \Q), \]
and the subgroup
\[G_J\defeq \left\{ g^s \in G \, \middle \vert \,
g_j^s =1 \;\; \forall j \in J\right\} \subseteq G.\] 
\end{dfn}   

Each $k$-dimensional face $\theta \preceq \tope =\tope^*_w$ 
is a $k$-dimensional simplex 
\[ \theta_J \defeq \conv(v_{i_0}, v_{i_1}, \ldots, v_{i_k}) \]
determined by a $(k+1)$-element subset 
$J \defeq \{i_0, i_1, \ldots, i_k \} \subseteq I = \{0, 1,\ldots, d \}$. 
We set $\widetilde{\theta}_J$ to be 
the $(k+1)$-dimensional simplex 
\[\widetilde{\theta}_J \defeq \conv(0,\widetilde{v}_{i_0}, 
\widetilde{v}_{i_1}, \ldots, \widetilde{v}_{i_k}) .  \]

\begin{prop} \label{volume}
Let $\emptyset \neq J \subseteq I$ be a non-empty subset of $I$. Then 
\begin{itemize}
\item[{\rm (i)}] $\vert \widetilde{M}_J/ M_J' \vert = 
\Vol_{\vert J \vert}(\widetilde{\theta}_J) $; 
\item[{\rm (ii)}] $\vert G_J\vert = n_J \defeq \ggT (w,  w_j \, \vert \, j \in J )$; 
\item[{\rm (iii)}] $\Vol_{\vert J \vert-1}(\theta_J)  = 
n_{\overline{J}}= \vert G_{\overline{J}} \vert,$
where $\overline{J} \defeq I \setminus J$.  
\end{itemize}
\end{prop}

\begin{proof} 
{\rm (i)} 
The normalized volume
$\Vol_{\vert J \vert}(\widetilde{\theta}_J)$ of $\widetilde{\theta}_J$ equals
the index of the sublattice $M'_J$ generated by the lattice vectors 
$\widetilde{v}_j \;( j \in J)$ in the $|J|$-dimensional lattice
$\widetilde{M}_J = \widetilde{M} \cap (M_J' \otimes \Q)$. 

\smallskip
{\rm (ii)} 
Assume without loss of generality $J = \{0,1,\ldots,k \}$. We 
set $u \defeq w/n_J \in \N$. Then one has  $u q_j = w_j/n_J  \in \Z$ for all
$j \in J $, \emph{i.e.}, $g^u \in G_J$.
This implies the inclusion $\langle g^u
\rangle \subseteq G_J$ and the inequality $n_J =  |\langle g^u
\rangle| \leq |G_J|$. In order to prove the opposite inclusion 
$G_J \subseteq \langle g^u
\rangle$, we set $u_j\defeq w_j/n_J \in \N$
for $j \in J$ and assume that  $g^s \in G_J$. 
Then $su_j/u=sq_j \in \Z$ for all $j \in J$, 
\emph{i.e.}, $u$ divides $su_j$ for all $j \in J$.
Since $\ggT(u, u_0, u_1, \ldots, u_k) =1$,  there
exist integers $a$ and $a_0, \ldots,a_k$ such that
\[ a u + \sum_{j \in J} a_j u_j = 1. \]
Therefore, $u$ divides  $s = asu +  \sum_{j \in J} a_j s u_j$,
\emph{i.e.}, $g^s \in  \langle g^u
\rangle$.

\smallskip
{\rm (iii)}
Note that the quotient $\widetilde{M}_J/ M_J'$ 
can be considered as a subgroup of the cyclic group
$\widetilde{M}/M' \cong \Z/w\Z$. Indeed, if we take the
homomorphism $\varphi\, :\,  \widetilde{M}_J \to \widetilde{M}/M'$ 
obtained from the embedding of $\widetilde{M}_J$ into $\widetilde{M}$,  then
the kernel of $\varphi$ is 
$\widetilde{M}_J \cap M' = \widetilde{M} \cap (M_J' \otimes \Q) \cap M' = M_J'$  
because $(M_J' \otimes \Q) \cap M' =  M_J'$. Furthermore, an element
$s(\widetilde{v} + M') \in \widetilde{M}/M'$ belongs to the
subgroup $\widetilde{M}_J/ M_J'$ if and only if 
the coefficients $s q_i $ in the equation
\[s \widetilde{v}  = \sum_{i =0}^d s q_i   \widetilde{v}_i \]
are integers for all $i \not\in J$.
By Proposition \ref{volume} (ii), the latter happens if and only if
$g^s \in G_{\overline{J}}$. This shows that  $\widetilde{M}_J/ M'_J \cong
G_{\overline{J}}$ and $\Vol_{k}(\theta_J) = \Vol_{k+1}(\widetilde{\theta}_{J}) =n_{\overline{J}}$.
\end{proof}

For the lattice simplex 
$\Delta = \Delta_{\w}^*= \conv(\{v_0, v_1, \ldots, v_d\})$ the dual polytope $\tope^*$ is 
the  $d$-dimensional rational simplex 
\[ \Delta^* = \tope_{\w} = \{ (u_0, u_1, \ldots, u_d) \in \R_{\geq 0}^{d+1} \, \vert \, 
\sum_{i=0}^d w_i u_i = w \}  - (1, 1,\ldots, 1).  \]
Let $J \subseteq I=\{0,1,\ldots,d\}$ be a (k+1)-element subset of $I$. 
Then we denote by $\sigma_J \defeq \sigma_{\theta_J}$ the normal cone
in the normal fan $\Sigma_{\tope}= \{\sigma_{J'} \, \vert \, \emptyset \subseteq J' \subseteq I\}$ 
corresponding to the face $\theta_J \preceq \tope$
of the lattice simplex $\tope$ with $\dim(\sigma_J) = d - \dim(\theta_J) = d-k$,
\emph{i.e.}, $\sigma_J$ is the cone generated by all inward-pointing facet normals of 
facets containing the face $\theta_J \preceq \tope$.

\begin{prop} \label{v-normal} 
Let $J \subseteq I=\{0,1,\ldots,d\}$ be a $(k+1)$-element subset of $I$. Then 
the normalized volume $\Vol_{d-k}(\sigma_J \cap \tope^*) $ of
the rational polytope $\sigma_J \cap \tope^*$ is given by
\[\Vol_{d-k}(\sigma_J \cap \tope^*) =
\frac{n_{\overline{J}}}{w} \prod_{i \in \overline{J}} \frac{1}{q_i} , \]
where $\overline{J} = I \setminus J$, $q_i = \frac{w_i}{w}$ $(i \in I)$,  
and $n_{\overline{J}} = \ggT(w,w_i \, \vert \, i \in \overline{J})$.
\end{prop}

\begin{proof}
\begin{figure}[t!]
	\centering
	\scalebox{0.6}{
	\includegraphics[width=\textwidth]{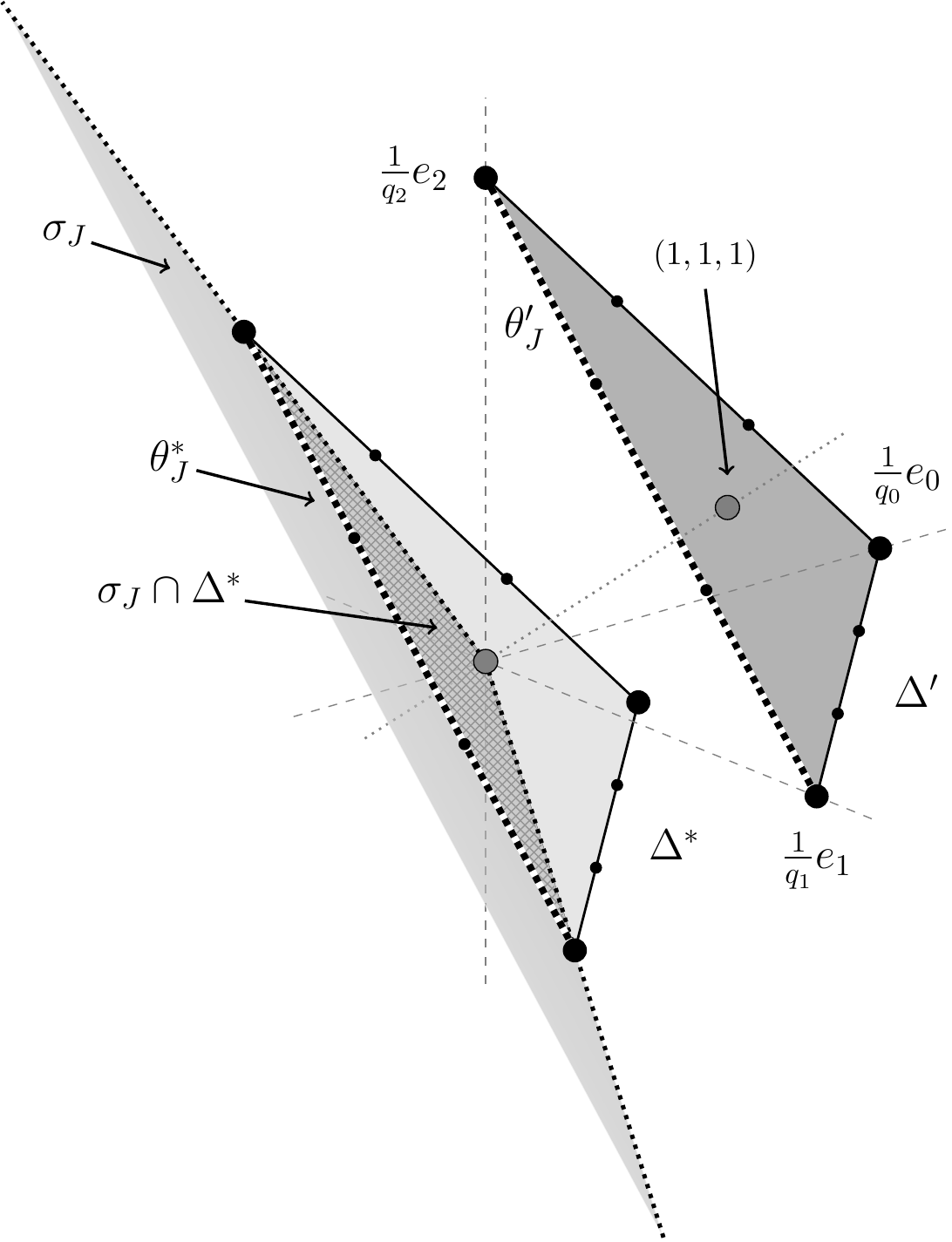}}
	\caption[]{{\bf Illustration of Proposition \ref{v-normal}.} 
	Grey coloured $2$-dimensional rational simplex $\tope'$ with unique interior lattice point $(1,1,1) \in \Z^3$ 
	and  vertices $\frac{1}{q_i}$ $(i \in I = \{0,1,2\})$ together with a $1$-dimensional dotted face 
	$\theta_J' \preceq \tope'$, \emph{i.e.}, $J =\{0\} \subseteq I$ and $d=2$.
	Light grey coloured shifted dual simplex $\tope^* = \tope' - (1,1,1)$ with unique interior lattice point 
	$(0,0,0) \in \Z^3$ together with a $1$-dimensional dotted shifted face $\theta_J^* \preceq \tope^*$. 
	The grey shaded associated $2$-dimensional normal cone $\sigma_J$ with two dotted rays corresponding 
	to the (rational) vertices of $\theta_J^*$, where $\theta_J = \{v_0\} \preceq \Delta$ is a vertex. 
	Moreover, the area $\sigma_J \cap \tope^*$ is crosshatched grey.} 
	\label{fig:VF1}
\end{figure}

Consider the $d$-dimensional rational simplex 
\[\tope' \defeq  \{ (u_0, u_1, \ldots, u_d) \in \R^{d+1} \, \vert \, u_i \geq 0 \; \; \forall i \in I, 
w_0 u_0 + w_1 u_1 + \ldots + w_d u_d = w\}.\]
Then the shifted simplex $\tope' - (1,1,\ldots, 1)$ is the dual simplex 
$\tope^* \subseteq N_\R$ of $\Delta$ because
\[ \tope^*= \{ (u_0, u_1, \ldots, u_d) \in \R^{d+1} \, \vert \, u_i \geq -1\; \;
\forall i \in I,  w_0 u_0 + w_1 u_1 + \ldots + w_d u_d = 0 \} \]
(Figure \ref{fig:VF1}).
For simplicity, we assume without loss of generality
$J = \{0,1,\ldots,k \}$, \emph{i.e.}, $\overline{J} = \{k+1, \ldots, d \}$ and
denote by $\{e_0, e_1, \ldots, e_d\}$ the standard basis of $\R^{d+1}$. 
The dual rational face $\theta_J^* \preceq \tope^* $ of $\tope^*$ as well as the 
shifted dual rational face 
$\theta_J' \defeq \theta_J^*+(1, 1,\ldots, 1)  \preceq \tope^* +(1, 1,\ldots, 1) =\tope'$ 
of the shifted simplex $\tope'$ have dimension $d-k-1$. 
Moreover, $\tope'$ has vertices $\frac{1}{q_i} e_i$ ($i \in I$) and 
$\theta_J^*+(1, 1,\ldots, 1) \preceq \tope'$ has vertices
$ \frac{1}{q_i} e_i  \;\; (i \in \overline{J})$ (Figure \ref{fig:VF1}). 

\begin{figure}
	\centering
	\scalebox{.4}{
	\includegraphics[width=\linewidth]{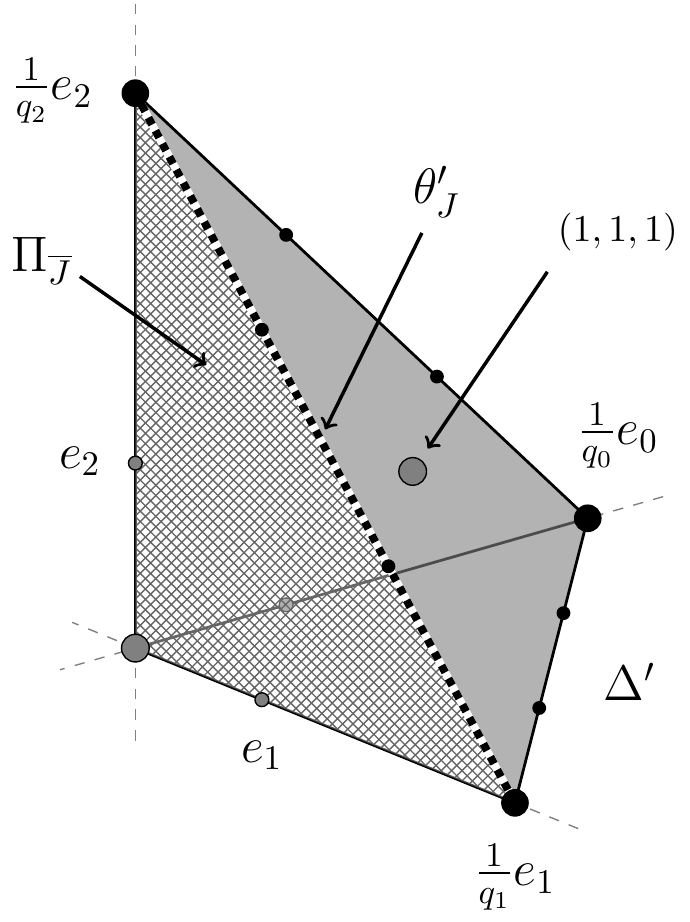}}
	\caption[]{{\bf Illustration for Proposition \ref{v-normal}.}  Shaded faces are occluded.
	Grey coloured $2$-dimensional rational simplex $\tope'$ with unique interior lattice point $(1,1,1) \in \Z^3$ 
	and vertices $\frac{1}{q_i}$ $(i \in I = \{0,1,2\})$ together with a $1$-dimensional dotted face 
	$\theta_J' \preceq \tope'$ with vertices $\frac{1}{q_i}$ $(i \in \overline{J})$, \emph{i.e.}, 
	$J =\{0\} \subseteq I$, $\overline{J} = \{1,2\}$, and $d=2$. The $2$-dimensional grey crosshatched 
	pyramid $\Pi_{\overline{J}}$ with vertex $(0,0,0) \in \Z^3$ and basis $\theta_J'$ .} 
	\label{fig:VF2}
\end{figure}

\smallskip
For $\vert J \vert = \vert I \vert = d+1$, \emph{i.e.}, $k=d$ and $\overline{J}= \emptyset$, we obtain 
$\theta_J = \tope$, $\sigma_J = \{0\}$, and  
\[\Vol_{d-k}(\sigma_J \cap \tope^*) 
=\Vol_{0}(\{0\})  =1 
=\frac{w}{w} \prod_{i \in \emptyset} \frac{1}{q_i} 
= \frac{n_{\overline{J}}}{w} \prod_{i \in \overline{J}} \frac{1}{q_i}.\]

For $\vert J \vert = d$, \emph{i.e.}, $k=d-1$ and $\overline{J}=\{w_d\}$, the associated 
dual rational face $\theta_J^* \preceq \tope^*$ has dimension $0$, 
\emph{i.e.}, $\theta_J^* \defeq \{v\}$ is a (rational) vertex and $P_v \defeq \conv(\{(0,0,\ldots,0),v\})$ 
a $1$-dimensional (rational) polytope. Therefore,
\begin{align*}
\Vol_{d-k}(\sigma_J \cap \tope^*) 
&=\Vol_{1}(P_v)  
=\frac{\ggT(w,w_d)}{w_d} 
= \frac{n_{\overline{J}}}{w} \prod_{i \in \overline{J}} \frac{1}{q_i}.
\end{align*}

Let $k\leq d-2$, \emph{i.e.}, $\vert J \vert \leq d-1$.
Moreover, let $ \Pi_{\overline{J}}$ be the pyramid with vertex $0 \in \R^{d+1}$ over
the shifted dual face $\theta_J' \preceq \tope'$, \emph{i.e.}, 
$\Pi_{\overline{J}} = \conv(\{0,  \frac{1}{q_i} e_i  \, \vert \, i \in \overline{J}\})$ with 
$\dim(\Pi_{\overline{J}})= d-k$ and
\[ \Vol_{d-k}( \Pi_{\overline{J}}) = \prod_{i  \in \overline{J}}  \frac{1}{q_i}  \]
(Figure \ref{fig:VF2}). 
The pyramid  $\Pi_{\overline{J}}$ is contained in the $(d-k)$-dimensional subspace
generated by $e_{k+1},\ldots, e_d$.  The basis of the pyramid is the simplex
$\theta_J'$ that belongs to a hyperplane in the affine linear subspace defined by the equation
\[ w_{k+1} u_{k+1} + \ldots + w_d u_d = w. \]
The integral distance between this hyperplane and the origin
equals $w/n_{\overline{J}}$, where $n_{\overline{J}} = \ggT(w, w_{k+1}, \ldots, w_d)$.
Therefore, we obtain
\[  \Vol_{d-k }(\sigma_J \cap \tope^*)= \Vol_{d-k-1}(\theta_J^*) =\Vol_{d-k-1 }(\theta_J') = 
\Vol_{d-k}( \Pi_{\overline{J}}) \cdot \frac{n_{\overline{J}}}{w} =
\frac{n_{\overline{J}}}{w} \prod_{i \in \overline{J}}  \frac{1}{q_i}. \]
\end{proof}

\begin{theo} \label{can-strEuler}
Let  $\w = (w_0, w_1, \ldots, w_d)$ be a weight vector with $\text{IP}$-property
and $Z_{\w} \subset \T_{\w}$ a non-degenerate affine
hypersurface defined by a Laurent 
polynomial $f_{\w}$ with Newton polytope $\Delta_{\w}^*$. 
Then the Zariski closure of $Z_{\w}$ 
in the $\Q$-Gorenstein toric Fano variety $\P^{\vee}({\w})$
is a $(d-1)$-dimensional Calabi-Yau variety $X_{\w}^*$ and  
\[ \chi_{\rm str}(X_{\w}^*) = \frac{1}{w} \sum_{\emptyset \subseteq J \subseteq I \atop |J| \geq 2} 
(-1)^{|J|} {n_{\overline{J}}^2} \prod_{i \in \overline{J}} \frac{1}{q_i},\]
where $\overline{J} = I \setminus J$ and $q_i=\frac{w_i}{w}$ $(i \in I)$.
\end{theo}

\begin{proof}
By assumption, $\tope(W)^\circ \cap \Z^{d+1} =(1,1,\ldots, 1)$, \emph{i.e.}, 
the lattice polytope $[\tope^*] =[(\tope_{\w}^*)^*]=\tope(W)-(1,1,\ldots,1)=\Delta'(W)$ is 
a $d$-dimensional canonical Fano polytope.
By Theorem \ref{Bat17223} and Remark \ref{ifpart},
the Zariski closure of the non-degenerate hypersurface $Z_{\w} \subset \T_{\w}$ 
in the toric variety $\P^{\vee}(\w)$ associated with the spanning fan $\Sigma^{\vee}(\w)$ 
of the canonical Fano polytope $[\tope^*]$ is a Calabi-Yau variety $X_{\w}^*$.
Therefore, we are able to apply Theorem  \ref{Bat17411} to compute 
the stringy Euler number $\chi_{\rm str}(X_{\w}^*)$ of $X_{\w}^*$ via
\begin{align*}  
\chi_{\rm str} (X_{\w}^*)  = \sum_{k=1}^{d} (-1)^{k-1} 
\sum_{ \theta \preceq \Delta \atop \dim (\theta) = k} 
\Vol_k (\theta) \cdot \Vol_{d-k} (\sigma_\theta \cap \Delta^*).
\end{align*}
Moreover, the $k$-dimensional faces of $\tope$ are 
simplices $\theta_J = \conv(\{v_j \, \vert \, j \in J\})$ pa\-ra\-me\-tri\-zed 
by subsets $J \subseteq I$ with $\vert J \vert =k+1$, 
where $k \geq 1$ if and only if $\vert J \vert \geq 2$.  

\smallskip
The  normalized volumes $\Vol_{\vert J \vert -1 }(\theta_J)$ and 
$\Vol_{d-\vert J \vert +1}(\sigma_J \cap \tope^*)$ 
have been computed  for every subset $J \subseteq I$ 
with $\vert J \vert \geq 2$ 
in Proposition \ref{volume} (iii) and Proposition \ref{v-normal}, respectively.  
By substitution, we obtain
\begin{align*}
\chi_{\str}(X_{\w}^*) 
&= \sum_{\vert J \vert \geq 2}  (-1)^{\vert J \vert -2} \Vol_{\vert J \vert -1 }(\theta_J) 
\cdot \Vol_{d-\vert J \vert +1}(\sigma_J \cap \tope^*)\\
&=\sum_{\emptyset \subseteq J \subseteq I \atop |J| \geq 2} 
(-1)^{|J|} \, n_{\overline{J}} \cdot \frac{ n_{\overline{J}} }{w}
\prod_{i \in {\overline{J}}} \frac{1}{q_i} 
= \frac{1}{w}\sum_{\emptyset \subseteq J \subseteq I \atop |J| \geq 2} 
(-1)^{|J|} \, n_{\overline{J}}^2  \prod_{i \in {\overline{J}}} 
\frac{1}{q_i} . 
\end{align*}
\end{proof}

%%%%%%%%%%%%%%%%%%%%%%%%%%%%%%%%%%%%%%%%%%%%%%%%%%
\section{Proof of Theorem \ref{vafemirrortest}}

Let $(w_0, w_1, \ldots, w_d) \in \Z_{>0}^{d+1}$ be a 
well-formed weight vector. We consider  the cyclic group
$G =  \langle g \rangle = \{ g^s\, | \, s \in \Z/w\Z \}  \subseteq SL(d+1, \C)$,  
whose generator $g = {\rm diag}(g_0, g_1, \ldots, g_d)$ acts linearly on $\C^{d+1}$
by the  diagonal matrix
\[{\rm diag}(e^{2\pi i q_0}, e^{2\pi i q_1}, \ldots, e^{2\pi i q_d}).\]

\smallskip
We aim at a combinatorial version of  
Vafa's Formula \eqref{vafa0}. This version has been proven 
in \cite[Formula (4.1), page 255]{GRY91}
for Calabi-Yau hypersurfaces in $4$-dimensional 
weighted projective spaces:

\begin{theo} \label{vafa-re}
Let $(w_0, w_1, \ldots, w_d) \in \Z_{>0}^{d+1}$ be as above. Then  
\[ \frac{1}{w}  \sum_{l,r=0}^{w-1}
\prod_{0 \leq i \leq d \atop lq_i,rq_i \in \Z}
\left(1 - \frac{1}{q_i} \right) 
= \frac{1}{w}  \sum_{\emptyset \subseteq
J \subseteq I \atop |J| \leq d-1} (-1)^{|J|}   \, n_J^2  \prod_{j \in J}
\frac{1}{q_j}. \]
\end{theo}

\begin{proof}
Obviously, we have 
\[\frac{1}{w}  \sum_{l,r=0}^{w-1}
\prod_{0 \leq i \leq d \atop lq_i,rq_i \in \Z}
\left( 1 - \frac{1}{q_i} \right) 
= \frac{1}{|G|}  \sum_{(g^l,g^r) \in G^2}
\prod_{0 \leq i \leq d \atop g_i^l=g_i^r =1}
\left(1 - \frac{1}{q_i} \right).\] 
First consider the trivial pair $(l,r)=(0,0) \in (\Z/w\Z)^2$. Then 
$g_i^l = g^r_i =1$ for all $i \in I$ and we obtain the product
\[ \prod_{i=0}^d
\left( 1 - \frac{1}{q_i} \right) =
\sum_{\emptyset \subseteq
J \subseteq I}  \prod_{j \in J}
\left( - \frac{1}{q_j} \right) = \sum_{\emptyset \subseteq
J \subseteq I}
(-1)^{|J|} \prod_{j \in J}
\frac{1}{q_j}  \]
that appears as one summand in Vafa's Formula \eqref{vafa0}.
For $s \in \Z/w\Z$, we define the subset $J_s \subseteq I$:
\[J_s \defeq  \{ i \in I \, \vert \, g_i^s =1\}. \] 
Then for any pair $(g^l,g^r) \in G^2$, we obtain 
\[ \prod_{0 \leq i \leq d \atop g_i^l = g_i^r =1}
\left(1 - \frac{1}{q_i} \right) =
\prod_{ i \in J_l \cap J_r} \left( 1 - \frac{1}{q_i} \right) =
\sum_{\emptyset \subseteq
J \subseteq J_l \cap J_r}
(-1)^{|J|} \prod_{j \in J} \frac{1}{q_j}  \]
and
\begin{align*}
\frac{1}{|G|}  \sum_{(g^l, g^r) \in G^2} 
\prod_{0 \leq i \leq d \atop g_i^l =g_i^r =1}
\left( 1 - \frac{1}{q_i} \right)  
&= \frac{1}{|G|}  \sum_{(g^l,g^r) \in G^2} \left(
\sum_{ \emptyset \subseteq
J \subseteq J_l \cap J_r }  (-1)^{|J|} \prod_{j \in J} \frac{1}{q_j} \right) \\
& = \frac{1}{|G|} \sum_{\emptyset \subseteq
J \subseteq I} \left( 
\sum_{(g^l, g^r) \in G_J^2}   (-1)^{|J|}  \prod_{j \in J}
\frac{1}{q_j} \right)   \\
&= \frac{1}{|G|}  \sum_{\emptyset \subseteq
J \subseteq I} (-1)^{|J|}  \, n_J^2  \prod_{j \in J}
\frac{1}{q_j}  ,
\end{align*}
where the last equality holds by Proposition \ref{volume} (ii).
Moreover, we note that $n_J =1$ if $|J| \in \{ d, d+1 \}$, since we assumed
\[\ggT (w_0, \ldots, w_{i-1}, w_{i+1}, \ldots, w_d) = 1 \; \; \forall i \in I. \]
Using the equality $\sum_{i =0}^d q_i = 1$, we obtain
\[ (-1)^{d+1}  \prod_{i=0}^d \frac{1}{q_i} 
+ (-1)^d \sum_{i=0}^d \prod_{j \in I \atop j \neq i}   \frac{1}{q_j} =0 \]
and this implies
\[ \frac{1}{w}  \sum_{l,r=0}^{w-1}
\prod_{0 \leq i \leq d \atop lq_i,rq_i \in \Z}
\left(1 - \frac{1}{q_i} \right)  =  \frac{1}{w}  \sum_{\emptyset \subseteq
J \subseteq I \atop |J| \leq d-1} (-1)^{|J|}  \, n_J^2  \prod_{j \in J} \frac{1}{q_j}. \]
\end{proof}

\begin{expar}[{\cite[page 1182]{Vaf89}}] 
Let $X_{15} \subset \P(1,2,3,4,5)$ be the quasi-smooth hypersurface of degree $15$ 
defined by the weighted homogeneous polynomial 
$W= z_4^3+z_4z_1^5+z_2^5+z_2z_3^3+z_0^{15}$. 
Then $q_0= \frac{1}{15}$, $q_1= \frac{2}{15}$, $q_2= \frac{1}{5}$, $q_3= \frac{4}{15}$, 
$q_4= \frac{1}{3}$, and 
\begin{align*}
\chi_{\rm orb}(X_w)
&=  \frac{1}{15} \sum_{l,r=0}^{14} \, \, \prod_{0 \leq i \leq 4 \atop lq_i,rq_i \in \Z}
\left( 1 - \frac{1}{q_i} \right)
= \frac{1}{15}  \sum_{\emptyset \subseteq
J \subseteq I \atop |J| \leq 3} (-1)^{|J|}   \, n_J^2  \prod_{j \in J} 
\frac{1}{q_j} \\  
& = \frac{1}{15}\cdot \left(225 - \frac{585}{4} + \frac{3375}{8} - \frac{19125}{8} \right) = -126.
\end{align*}
\end{expar}

Now we prove the main theorem of our paper:

\begin{proof}[Proof of Theorem \ref{vafemirrortest}] 
By assumption, $\w = (w_0, w_1, \ldots, w_d) $ is an arbitrary 
weight vector with ${\rm IP}$-property and $Z_{\w} \subset \T_{\w}$ a 
non-degenerate affine hypersurface 
defined by a Laurent polynomial $f_{\w}$ with Newton polytope 
$\Delta_{\w}^*$. Applying Theorem \ref{can-strEuler}, the Zariski closure of $Z_{\w}$ 
 in the $\Q$-Gorenstein toric Fano variety $\P^{\vee}({\w})$
is a Calabi-Yau variety $X_{\w}^*$ and
\begin{align*}
\chi_{\str}(X_{\w}^*) = \frac{1}{w}\sum_{\emptyset \subseteq J \subseteq I \atop |J| \geq 2} 
(-1)^{|J|} \, n_{\overline{J}}^2  \prod_{i \in {\overline{J}}} 
\frac{1}{q_i} . 
\end{align*}
Since $|J| + |\overline{J}| = |I| = d+1$, we obtain 
\[ \chi_{\str}(X_{\w}^*) = 
(-1)^{d-1}    \frac{1}{w} \sum_{\emptyset \subseteq J \subseteq I \atop |\overline{J}| \leq d-1} 
(-1)^{|\overline{J}|} \, n_{\overline{J}}^2  \prod_{i \in {\overline{J}}}  \frac{1}{q_i}
=(-1)^{d-1}    \frac{1}{w}  \sum_{\emptyset \subseteq J \subseteq I \atop |J| \leq d-1} 
(-1)^{|J|}   n_J^2  \prod_{j \in J}  
\frac{1}{q_j} \]
and Theorem \ref{vafa-re} implies
\[ \chi_{\str}(X_{\w}^*)  = (-1)^{d-1}  \frac{1}{w}  
\sum_{\emptyset \subseteq J \subseteq I \atop |J| \leq d-1} 
(-1)^{|J|}   n_J^2  \prod_{j \in J}  \frac{1}{q_j}
= (-1)^{d-1}  \frac{1}{w}  \sum_{l,r=0}^{w-1}
\prod_{0 \leq i \leq d \atop lq_i,rq_i \in \Z}
\left( 1 - \frac{1}{q_i} \right).\]

\smallskip
If  $\w= (w_0, w_1, \ldots, w_d)$ is a transverse weight vector, 
Vafa's Formula \eqref{vafa0} yields
\[\chi_{\rm orb}(X_w) =  \frac{1}{w}  \sum_{l,r=0}^{w-1}
\prod_{0 \leq i \leq d \atop lq_i,rq_i \in \Z}\left( 1 - \frac{1}{q_i} \right). \]
Since $\w$ has  
${\rm IP}$-property  
\cite[Lemma 2]{Ska96}, 
a combination of the last two equations implies
\[  \chi_{\rm str}(X_{\w}^*) = (-1)^{d-1} \chi_{\rm orb}(X_w). \]
\end{proof}

%%%%%%%%%%%%%%%%%%%%%%%%%%%%%%%%%%%%%%%%%%%%


\begin{thebibliography}{CdlOGP91}

\bibitem[ACG16]{ACG16}
Michela Artebani, Paola Comparin, and Robin Guilbot, \emph{Families of
  {C}alabi-{Y}au hypersurfaces in {$\Bbb{Q}$}-{F}ano toric varieties}, J. Math.
  Pures Appl. (9) \textbf{106} (2016), no.~2, 319--341.

\bibitem[ALR07]{ALR07}
Alejandro Adem, Johann Leida, and Yongbin Ruan, \emph{Orbifolds and stringy
  topology}, Cambridge Tracts in Mathematics, vol. 171, Cambridge University
  Press, Cambridge, 2007.

\bibitem[Bat94]{Bat94}
Victor~V. Batyrev, \emph{{Dual polyhedra and mirror symmetry for Calabi-Yau
  hypersurfaces in toric varieties}}, Journal of Algebraic Geometry \textbf{3}
  (1994), no.~3, 493--535.

\bibitem[Bat98]{Bat98}
\bysame, \emph{{Stringy Hodge numbers of varieties with Gorenstein canonical
  singularities}}, Integrable systems and algebraic geometry (Kobe/Kyoto 1997)
  (1998), 1--31, World Sci. Publ., River Edge, NJ.

\bibitem[Bat17]{Bat17}
Victor Batyrev, \emph{The stringy {E}uler number of {C}alabi-{Y}au
  hypersurfaces in toric varieties and the {M}avlyutov duality}, Pure Appl.
  Math. Q. \textbf{13} (2017), no.~1, 1--47.

\bibitem[BB96]{BB96}
Victor~V. Batyrev and Lev~A. Borisov, \emph{{Mirror duality and
  string-theoretic Hodge numbers}}, Inventiones Mathematicae \textbf{126}
  (1996), no.~3, 183--203.

\bibitem[BD96]{BD96}
Victor~V. Batyrev and Dimitrios~I. Dais, \emph{{Strong McKay correspondende,
  string-theoretic Hodge numbers and mirror symmetry}}, Topology \textbf{35}
  (1996), no.~4, 901--929.

\bibitem[BK16]{BK16}
Gavin Brown and Alexander Kasprzyk, \emph{Four-dimensional projective orbifold
  hypersurfaces}, Exp. Math. \textbf{25} (2016), no.~2, 176--193.

\bibitem[BKS19]{BKS19}
Victor~V. Batyrev, Alexander~M. Kasprzyk, and Karin Schaller, \emph{{On the
  Fine Interior of Three-dimensional Canonical Fano Polytopes}},
  arXiv:1911.12048.

\bibitem[Bor13]{Bor13}
Lev~A. Borisov, \emph{Berglund-{H}\"{u}bsch mirror symmetry via vertex
  algebras}, Comm. Math. Phys. \textbf{320} (2013), no.~1, 73--99.

\bibitem[BS19]{BS19}
Victor~V. Batyrev and Karin Schaller, \emph{Stringy {$E$}-functions of
  canonical toric {F}ano threefolds and their applications}, Izv. Ross. Akad.
  Nauk Ser. Mat. \textbf{83} (2019), no.~4, 26--49.

\bibitem[CdlOGP91]{CdlOGP91}
Philip Candelas, Xenia~C. de~la Ossa, Paul~S. Green, and Linda Parkes, \emph{A
  pair of {C}alabi-{Y}au manifolds as an exactly soluble superconformal
  theory}, Nuclear Phys. B \textbf{359} (1991), no.~1, 21--74.

\bibitem[CdlOK95]{CdlOK95}
Philip Candelas, Xenia de~la Ossa, and Sheldon Katz, \emph{Mirror symmetry for
  {C}alabi-{Y}au hypersurfaces in weighted {${\bf P}_4$} and extensions of
  {L}andau-{G}inzburg theory}, Nuclear Phys. B \textbf{450} (1995), no.~1-2,
  267--290.

\bibitem[CG11]{CG11}
Alessio Corti and Vasily Golyshev, \emph{{Hypergeometric equations and weighted
  projective spaces}}, Science China Mathematics \textbf{54} (2011), no.~8,
  1577--1590.

\bibitem[CLS90]{CLS90}
Philip Candelas, Monika Lynker, and Rolf Schimmrigk, \emph{Calabi-{Y}au
  manifolds in weighted {${\bf P}_4$}}, Nuclear Phys. B \textbf{341} (1990),
  no.~2, 383--402.

\bibitem[CLS11]{CLS11}
David Cox, John Little, and Henry Schenck, \emph{{Toric Varieties}}, American
  Mathematical Society, 2011.

\bibitem[CR04]{CR04}
Weimin Chen and Yongbin Ruan, \emph{A new cohomology theory of orbifold}, Comm.
  Math. Phys. \textbf{248} (2004), no.~1, 1--31.

\bibitem[CR10]{CR11}
Alessandro Chiodo and Yongbin Ruan, \emph{{LG/CY correspondence: the state
  space isomorphism}}, Adv. Math. \textbf{227} (2010).

\bibitem[CR18]{CR18}
Emily Clader and Yongbin Ruan, \emph{Mirror symmetry constructions}, B-model
  {G}romov-{W}itten theory, Trends Math., Birkh\"{a}user/Springer, Cham, 2018,
  pp.~1--77.

\bibitem[GP90]{GP90}
Brain~R. Greene and Moshe~R. Plesser, \emph{Duality in {C}alabi-{Y}au moduli
  space}, Nuclear Phys. B \textbf{338} (1990), no.~1, 15--37.

\bibitem[GRY91]{GRY91}
Brain~R. Greene, Shi-Shyr Roan, and Shing-Tung Yau, \emph{Geometric
  singularities and spectra of {L}andau-{G}inzburg models}, Comm. Math. Phys.
  \textbf{142} (1991), no.~2, 245--259.

\bibitem[HV00]{HV00}
Kentaro Hori and Cumrun Vafa, \emph{{Mirror Symmetry}},
  {arXiv:hep-th/0002222v3}.

\bibitem[Kas10]{Kas10}
Alexander~M. Kasprzyk, \emph{{Canonical toric Fano threefolds}}, Canadian
  Journal of Mathematics \textbf{62} (2010), no.~6, 1293--1309.

\bibitem[Kel13]{Kel13}
Tyler~L. Kelly, \emph{Berglund-{H}\"{u}bsch-{K}rawitz mirrors via {S}hioda
  maps}, Adv. Theor. Math. Phys. \textbf{17} (2013), no.~6, 1425--1449.

\bibitem[Kho78]{Kho78}
Askold~G. Khovanski\u\i, \emph{Newton polyhedra and the genus of complete
  intersections}, Funktsional. Anal. i Prilozhen. \textbf{12} (1978), no.~1,
  51--61.

\bibitem[Kra10]{Kra10}
Marc Krawitz, \emph{F{JRW} rings and {L}andau-{G}inzburg mirror symmetry},
  ProQuest LLC, Ann Arbor, MI, 2010, Thesis (Ph.D.)--University of Michigan.

\bibitem[KS92]{KS92}
Maximilian Kreuzer and Harald Skarke, \emph{No mirror symmetry in
  {L}andau-{G}inzburg spectra!}, Nuclear Phys. B \textbf{388} (1992), no.~1,
  113--130.

\bibitem[KS94]{KS94}
Albrecht Klemm and Rolf Schimmrigk, \emph{Landau-{G}inzburg string vacua},
  Nuclear Phys. B \textbf{411} (1994), no.~2-3, 559--583.

\bibitem[KS98a]{KS98b}
Maximilian Kreuzer and Harald Skarke, \emph{Calabi-{Y}au {$4$}-folds and toric
  fibrations}, Journal of Geometry and Physics \textbf{26} (1998), no.~3-4,
  272--290.

\bibitem[KS98b]{KS98a}
Maximilian Kreuzer and Harald Skarke, \emph{{Classification of reflexive
  polyhedra in three dimensions}}, Advances in Theoretical and Mathematical
  Physics \textbf{2} (1998), no.~4, 853--871.

\bibitem[KS00]{KS00}
\bysame, \emph{{Complete classification of reflexive polyhedra in four
  dimensions}}, Adv. Theor. Math. Phys. \textbf{4} (2000).

\bibitem[LSk99]{LSW99}
Monika Lynker, Rolf Schimmrigk, and Andreas~Wi\ss kirchen,
  \emph{Landau-{G}inzburg vacua of string, {M}- and {F}-theory at {$c=12$}},
  Nuclear Phys. B \textbf{550} (1999), no.~1-2, 123--150.

\bibitem[Mav11]{Mav11}
Anvar~R. Mavlyutov, \emph{{Mirror Symmetry for Calabi-Yau complete
  intersections in Fano toric varieties}}, arXiv:1103.2093.

\bibitem[OR93]{OR93}
Kaoru Ono and Shi-Shyr Roan, \emph{Vafa's formula and equivariant
  {$K$}-theory}, J. Geom. Phys. \textbf{10} (1993), no.~3, 287--294.

\bibitem[Rei80]{Rei80}
Miles Reid, \emph{Canonical {$3$}-folds}, Journ\'{e}es de {G}\'{e}ometrie
  {A}lg\'{e}brique d'{A}ngers, {J}uillet 1979/{A}lgebraic {G}eometry, {A}ngers,
  1979, Sijthoff \& Noordhoff, Alphen aan den Rijn---Germantown, Md., 1980,
  pp.~273--310.

\bibitem[Roa90]{Roa90}
Shi-Shyr Roan, \emph{On {C}alabi-{Y}au orbifolds in weighted projective
  spaces}, Internat. J. Math. \textbf{1} (1990), no.~2, 211--232.

\bibitem[Ska96]{Ska96}
Harald Skarke, \emph{Weight systems for toric {C}alabi-{Y}au varieties and
  reflexivity of {N}ewton polyhedra}, Modern Phys. Lett. A \textbf{11} (1996),
  no.~20, 1637--1652.

\bibitem[SchS19]{SchS19}
Friedrich Sch\"{o}ller and Harald Skarke, \emph{All weight systems for
  {C}alabi-{Y}au fourfolds from reflexive polyhedra}, Comm. Math. Phys.
  \textbf{372} (2019), no.~2, 657--678.

\bibitem[Vaf89]{Vaf89}
Cumrun Vafa, \emph{{String vacua and orbifoldized {LG} models}}, Modern Phys.
  Lett. A \textbf{4} (1989), no.~12, 1169--1185.

\bibitem[Wit92]{Wit92}
Edward Witten, \emph{Mirror manifolds and topological field theory}, Essays on
  mirror manifolds, Int. Press, Hong Kong, 1992, pp.~120--158.

\end{thebibliography}
\end{document}